\documentclass[12pt,a4]{article}

\usepackage{amsmath,graphicx,amsthm,amssymb,mathrsfs}

\newcommand{\exn}{{\bf E}}
\newcommand{\pr}{{\bf P}}

\newcommand{\R}{\mathbb{R}}
\newcommand{\Z}{\mathbb{Z}}

\newcommand{\Xc}{\mathcal{X}}
\newcommand{\Mc}{\mathcal{M}}
\newcommand{\Ic}{\mathcal{I}}
\newcommand{\Nc}{\mathcal{N}}
\newcommand{\Ac}{\mathcal{A}}
\newcommand{\al}{\alpha}

\newcommand{\bal}{\boldsymbol{\alpha}}
\newcommand{\Con}{{\bf [C$_{\boldsymbol{1}}$]}}
\newcommand{\Conn}{{\bf [C$_{\boldsymbol{2}}$]}}
\newcommand{\la}{\lambda}

\newcommand{\si}{\sigma}
\newcommand{\de}{\delta}
\newcommand{\De}{\Delta}

\newcommand{\deq}{\stackrel{d}{=}}

\newcommand{\Va}{\mbox{\rm Var}\, }

\newcommand{\ind}{{\bf 1}}

\newcommand{\todi}{\stackrel{d}{\longrightarrow}}

\newtheorem{rema}{Remark}
\newtheorem{theo}{Theorem}

\newtheorem{lemo}{Lemma}
\newtheorem{coro}{Corollary}

\title{Limit theorems for record indicators in threshold  $F^\al$-schemes}
\author{Patrick He$^1$ and Konstantin Borovkov$^2$ }

\begin{document}
	\maketitle

\footnotetext[1]{School of Mathematics and Statistics, The University of Melbourne, Parkville 3010, Australia.}

\footnotetext[2]{School of Mathematics and Statistics, The University of Melbourne, Parkville 3010, Australia; e-mail: borovkov@unimelb.edu.au.}

	\begin{abstract}In Nevzorov's $F^\alpha$-scheme, one deals with a sequence of independent random variables whose distribution functions are all powers of a common continuous distribution function. A key  property of the $F^\alpha$-scheme is that the record indicators for such a sequence are independent. This allows one to obtain several important limit theorems for the total number of records in the sequence up to time $n\to\infty$. We extend these theorems to a much more general class of sequences of random variables obeying a ``threshold $F^\alpha$-scheme" in which the distribution functions of the variables are close to the powers of a common $F$ only in their right tails, above certain non-random non-decreasing threshold  levels. Of independent interest is the characterization of the growth rate for extremal processes that we derived in order to be able to verify the conditions of our main theorem. We also establish the asymptotic pair-wise independence of record indicators in a special case of threshold $F^\alpha$-schemes.
		
		
	\end{abstract}

\section{Introduction and main results}

 Let $\boldsymbol{X}:=\{X_n\}_{n\ge 1}$
be a sequence of random variables (r.v.'s) 	on a common probability space, $M_n:=  \bigvee_{1\le k\le n}X_k,$ $n\ge 1,$ be the sequence  of the partial maxima of these r.v.'s, and   $I_1:=1,$
\[
\quad I_n:= \ind (X_n> M_{n-1}), \quad n\ge 2,
\]
the (upper) record indicators for $\boldsymbol{X}$. Denote by  $N_n:=\sum_{k=1}^n I_k$ the total number of records  in $\boldsymbol{X}$ up to time~$n\ge 1.$   Apart from the natural motivation related to the theory of records, the study of the distribution of $N_n$  is also of interest for other applications, e.g.\ in connection with the secretary problem~\cite{Pf89} or for the linear search problem of the maximum element in a field of $n$ entries, where $N_n$  denotes
the number of re-storages during the procedure (for details see e.g.~\cite{Ke85, Pf91}). Another field of relevance one might mention is the average-case analysis of the simplex method in linear programming~\cite{DePf87,Ro82}.

For an outline of the history of the theory of records, we refer the reader to Appendix~1 in~\cite{Ne00}. For further  detail, see also~\cite{Ne88, NeAh15} and the bibliography therein. The  most studied case is, of course, when the $X_n$'s are independent and identically distributed (i.i.d.) with a continuous distribution function (d.f.)~$F$. By the Dwass--R\'enyi theorem (see e.g.\ Ch.~3 in~\cite{NeAh15}),    the record indicators for such sequences~$\boldsymbol{X}$ are jointly independent with $\pr (I_n=1)= n^{-1},$ $n\ge 1.$ The independence property enables one to establish a number of limit theorems  for the distribution of~$N_n$, including the Poisson and normal approximations, the respective convergence rates in the uniform norm for the  d.f.'s  being the rather slowly decaying  $O(\ln^{-3/2}n)$ and $O(\ln^{-1/2}n)$ as $n\to\infty$. One should note, however, that the former approximation can be dramatically improved by a remarkably  simple ``adjusted Poisson approximation" from~\cite{BoPf96} with a convergence rate of the form $O(n^{-2})$ (with an explicit bound for the constant).

It was discovered in~\cite{Ne81, Ne85} that the important independence property for the record indicators also holds  for a special class of scenarios nowadays referred to  as the (Nevzorov)  ``$F^\al$-scheme". In that scheme, the $X_n$'s are independent r.v.'s following the respective distribution functions   $F^{\al_n}$, $n\ge 1,$ with   a common continuous d.f.~$F,$     $\bal:=\{\al_n \}_{n\ge 1}$ being an arbitrary positive sequence  (a special  case of the $F^\al$-scheme where the $\al_n$'s are integers had been earlier analyzed in~\cite{Ya75}). This scheme plays an important role in the present paper, so it will be convenient for us to adopt a special notation for the related r.v.'s to distinguish them from the ones for the original~$\boldsymbol{X}$: we will use $\Xc_n, \Mc_n, $  $\Ic_n$ and $\Nc_n,$  respectively, for the independent r.v.'s in the $F^\al$-scheme with some~$F$ and  $\bal$ (setting $\boldsymbol{\Xc}:=\{\Xc_n\}_{n\ge 1}$), their partial maxima, record indicators and   record  counts.  It turned out that the record  indicators $\Ic_1,\Ic_2,\ldots $ in the case of  the $F^\al$-scheme form a sequence of independent r.v.'s with
\begin{equation}
\label{a/s}
p_n:=\pr (\Ic_n=1)=\frac{\al_n}{s_n },\quad s_n :=\sum_{k=1}^n \al_k,\quad n\ge 1
\end{equation}
(see e.g.\ p.~217 in~\cite{NeAh15} and~\cite{BoPf95}). This result was also demonstrated  in~\cite{BaRe87}  using a natural embedding of the sequence of partial maxima $\Mc_n$ in the  so-called extremal process, yielding the  additional fact that
\begin{equation}
\label{I_M}
\mbox{$ \Mc_n $  is   independent of $\Ic_1,\ldots,\Ic_n$, \ $n\ge 1.$ }
\end{equation}

Moreover, it  turned out that the $F^\al$-scheme is basically  the only situation with   the original  $X_n$'s being independent  where the r.v.'s $I_n$ and $M_n$  are   independent of each other for any $n\ge 1$  (Theorem~3 in~\cite{BoPf95}).

The independence of the record indicators in the $F^\al$-scheme enables one to establish a number of asymptotic results for the behavior of~$\Nc_n$ as $n\to\infty.$ We will summarize  the key ones. Note that $E_n:=\exn \Nc_n=\sum_{k=1}^n p_k,$ $V_n:=\Va(\Nc_n)=\sum_{k=1}^n p_k (1-p_k),$ $n\ge 1.$

\begin{enumerate}
	\item[(A1)] First of all,
  $\Nc_n\to \infty $ a.s.\ as $n\to \infty $ iff condition
\medskip

\Con~~$\lim_{n\to \infty }s_n =\infty$

\medskip\noindent
is met~\cite{Ne85}.

This condition will be assumed to be satisfied throughout  this paper (and, in particular, in assertions  (A2)--(A4) in this list).  Note that under condition~\Con\ both $E_n$ and $V_n$ tend to infinity as $n\to\infty$.

\item[(A2)]   $\lim_{n\to\infty}\Nc_n/E_n=1$ a.s.\ (Theorem~1 in~\cite{DoKlPaSt13}).

\item[(A3)]   $\lim_{n\to\infty}(\Nc_n-E_n)/\ln s_n =0$ a.s. If $p_n\to 0 $ then  $\lim_{n\to\infty} \Nc_n /\ln s_n =1$ a.s., while  if $p_n\to 1 $ then  $\lim_{n\to\infty} \Nc_n /\ln s_n =0$ a.s.~\cite{DoKlSt15}.

\item[(A4)] If $\lim_{n\to \infty } V_n<\infty$ then $\Nc_n-E_n$ converges a.s.\ to a proper r.v.\ as $n\to\infty,$ while if $\lim_{n\to \infty } V_n=\infty$ then $V_n^{-1/2} (\Nc_n-E_n) \todi Z\sim N(0,1)$ (Theorem~7 in~\cite{DoKlPaSt13}).

\end{enumerate}

In the general case, even when the $X_n$'s are independent,  the nature of dependence between record indicators  is very complicated and difficult to describe,   so obtaining results similar to (A1)--(A4) appears to be a rather hard task.
We will note, however,    the remarkable results on coupling of the  record times and values for a class of  strictly stationary sequences~$\boldsymbol{X}$ with ``time-shifted" record times and values, respectively, for sequences of i.i.d.\ r.v.'s with the same univariate marginal distributions as for~$\boldsymbol{X}$, see~\cite{Ha87, HaMaNePu98} and the references in the latter paper. Such couplings imply, in particular, that the above assertions (A1)--(A.4) will hold for such  stationary sequences as well, with the quantities $E_n,$ $V_n$ corresponding to the i.i.d.\ case.
Any advances  extending the limit theory for records beyond such special cases remain to  be  highly  desirable.

The key observation that led us to writing this paper was, roughly speaking,
that, for large~$n$ values, for the observation $X_n$ to be a record it needs to be ``large" as it has to exceed the previous  partial maximum value $M_{n-1}$ that is likely to  already be ``large". Therefore, if, for $n\ge 1,$ the d.f.\ of  $X_n$   is equal (or close) to the respective $F^{\al_n }$ from an $F^{\al}$-scheme  $\boldsymbol{\Xc}$ in its ``right tail" only, i.e.\ above a certain non-random threshold $\ell_n$, then one can still expect the record indicators to display an asymptotic behavior close to that of the record indicators for~$\boldsymbol{\Xc}$. Moreover, one can relax the independence assumption as well, since the nature of dependence between the $X_n$'s on  the event where the observations are unlikely to be records would be of little relevance. The respective result, stated as Theorem~\ref{th1} below, is the first main contribution of this paper. The key condition in the theorem is that  eventually $\Mc_n> \ell_n$ (meaning that $\Mc_n > \ell_n$ for all sufficiently large~$n$ and abbreviated as ``$\Mc_n> \ell_n$ ev.") a.s.

Our second contribution  concerns the question of when that key condition is met. We establish  the a.s.\ rate of growth of the sequence $\{\Mc_n\}$,  providing    criteria for
$\pr(\Mc_n >  \ell_n \ {\rm ev.})=1$. This result is stated in  Corollary~\ref{co2} of Theorem~\ref{th2}, the latter dealing with a similar question for the extremal process. This result   is an extension of the   work in~\cite{Kl84, Kl85} on such criteria  in the case of i.i.d.\ sequences  $\boldsymbol{X}$  and   is of independent interest.

The third main contribution of this paper  is Theorem~\ref{th3} below, which   establishes the uniform asymptotic pairwise independence of the record indicators in the special case of the threshold $F^\al$-scheme where $\boldsymbol{X}$ consists of independent  r.v.'s and there is a common threshold $\ell_n\equiv \ell$.

Now we will give a formal definition of the above-mentioned threshold $F^\al$-scheme and state our main results.

Denote by $F_n$ the d.f.\ of $X_n $ and by
\[
F_{n|x_1, \ldots, x_{n-1}}(x_n):=
\pr (X_n\le x_n|X_1=x_1, \ldots, X_{n-1}= x_{n-1}), \quad x_j\in \R, \quad 1\le j\le n,
\]
the   conditional d.f.\ of $X_n$ given the values of the $n-1$  ``preceding observations". 
We will use the standard notation
\[
G^\leftarrow (u):= \inf\{x\in\R: G(x)\ge u\}, \quad u\in (0,1),
\]
for the generalized inverse function of the d.f.~$G.$
Finally, we will denote by $G|_x$ the restriction of the distribution $G$ to the half-line $(x,\infty)$: $dG|_x(y):=\ind (y>x)\, dG(y),$ $y\in \R.$

By a threshold $F^\al$-scheme we will mean any sequence $\boldsymbol{X}$ of r.v.'s satisfying the following condition:

\medskip

\Conn~~{\em There exist a continuous d.f.\ $F$,  a positive sequence $\bal$ satisfying \Con\ and a  non-decreasing real  sequence $\{\ell_n\}_{n\ge 1}$ such that$:$
	\begin{enumerate}
		\item[\rm (i)]  $F_n^{\leftarrow} (1-)\le F^{\leftarrow} (1-) ,$ $n\ge 1;$

		\item[\rm (ii)]
$
F_{n|x_1, \ldots, x_{n-1}}(x_n)
   = F_{n }(x_n) $ for $x_n \ge  \ell_n,$ $  x_j\in \R,$ $ 1\le j< n,$ $n\ge 2;$

\item[\rm (iii)] $\sum_{n\ge 1}\de_n< \infty,$ where
$$
\de_n:=  d_{TV} (F_n|_{ \ell_n }, F^{\al_n}|_{ \ell_n }):=
 \int_{(\ell_n,\infty)} |d(F_n-  F^{\al_n})| 
 $$
denotes the total variation distance between the restrictions of the distributions $F_n$ and $F^{\al_n}$ to the half-line $(\ell_n,\infty)$, and

\item[\rm (iv)]
$\pr (\Mc_n>\ell_n{\rm \ ev.})=1$ for the sequence $\{\Mc_n\}_{n\ge 1}$ of the partial  maxima in the  $F^\al$-scheme $\boldsymbol{\Xc}= \{\Xc_n\}_{n\ge 1}$ specified by $F$ and~$\bal.$
\end{enumerate}}

\medskip  That is, for any $n\ge 2,$ on the event $\{X_n>\ell_n\}$ the r.v.\ $X_n$ is independent of the observations $X_1, \ldots, X_{n-1}$ and the restriction of the distribution of $X_n$ to  the half-line $(\ell_n,\infty)$ is close (in the total variation sense) to the restriction to   $(\ell_n, \infty)$ of the law of  the $n$th element of an $F^{\al}$-scheme that has the property that, with probability~1, its partial maxima $\Mc_n$ will eventually lie above the threshold values~$\ell_n$.

Parts~(ii) and~(iii) of the above condition may seem quite  strong. In fact, the distributional properties of the sequence of record indicators are very sensitive to changes in the  distribution of the original sequence. Therefore  it should not be surprising that, to obtain results at the level of (A1)--(A4), one would need to make  relatively strong assumptions about~$\boldsymbol{X}.$ What we would like to stress, though, is  that those assumptions only need to be made about  the right tails of the (conditional) distributions of the elements of~$\boldsymbol{X}.$

\begin{rema}
\rm Note    that the assumption that $\{\ell_n\}_{n\ge 1}$ is a non-decreasing sequence does not actually restrict the  generality, cf.~Remark~\ref{re1} below. The purpose of part~(i) of   condition~\Conn\ is to ensure that the partial  maxima $M_n$ of the r.v.'s in the original sequence~$\boldsymbol{X}$ cannot take values ``beyond the reach" of the maxima in the $F^\al$-scheme~$\boldsymbol\Xc$. Necessary and sufficient conditions for part~(iv) to hold are established in our Corollary~\ref{co2} below.
\end{rema}

A simple special case where condition \Conn\ is satisfied is a sequence $\boldsymbol{X}$  of independent r.v.'s such that $F_n(x) = F^{\al_n}(x)$ for $x> \ell_n$ (so that $\de_n\equiv 0$), $n\ge 1$, for some continuous d.f.~$F$ and positive sequence~$\bal.$  A more interesting example of a situation where that condition can be satisfied is when $X_n= V_n \vee Y_n,$ $n\ge 1,$ under the assumptions that $  \{V_n\}_{n\ge 1}$ is an arbitrary   sequence of bounded from above r.v.'s ($\pr (V_n\le \ell_n)=1$  for all $n\ge 1$), whereas the r.v.'s $Y_n$   are independent of each other and of $  \{V_n\}_{n\ge 1}$. As for further  conditions on the $Y_n$'s, it suffices to  assume that there exists an $F^\al$-scheme satisfying conditions~\Con, \Conn(iv) such that the total variation distances $\De_n$ between the laws of $Y_n,$ $n\ge 1,$ and the respective $F^{\al_n}$ are such that $\sum_{n\ge 1}\De_n<\infty.$

Our first result establishes the existence of a coupling of the sequences $\{(M_n, I_n)\}_{n\ge 1}$ and $\{(\Mc_n, \Ic_n)\}_{n\ge 1},$ the latter corresponding to the $F^\al$-scheme $\boldsymbol{\Xc}$ from~\Conn.

\begin{theo}\label{th1}
If $\boldsymbol{X}$ satisfies condition \Conn\  then one can construct the sequences $\boldsymbol{X}$ and $\boldsymbol{\Xc}$ on a common probability space so that there exists a random time $T<\infty$ a.s.\ such that  $(M_n, I_n)=(\Mc_n, \Ic_n)$ for all $n\ge T.$
\end{theo}

The following corollary is a  straightforward  consequence of the coupling established in Theorem~\ref{th1} and the observation that $E_n,$ $V_n\to\infty$ as $n\to \infty$ under condition~\Con.

\begin{coro}\label{co1}
Under the condtions of Theorem~\ref{th1},   the above  assertions {\rm (A1)--(A4)} concerning the limiting behavior of $\Nc_n$ as $n\to\infty$ remain true if we replace in them $\Nc_n$ with $N_n$, the definitions of $E_n, V_n$ staying unchanged.
\end{coro}

One of the key components of  \Conn\ is the condition that $\pr (\Mc_n >\ell_n\ {\rm ev.})=1.$ We will obtain a criterion for that relation  as a consequence of our Theorem~\ref{th2} on the growth rate of the  extremal process. The   assertion that theorem builds on is the criterion for the partial maxima of i.i.d.\ r.v.'s which was derived in~\cite{Kl84, Kl85} and is stated in~\eqref{Klass} below. For comments on the history of the   problem  on the growth rate for the partial  maxima in the i.i.d.\ case see~\cite{Kl84}. For an alternative martingale-based proof of  criterion~\eqref{Klass}  see~\cite{Go87}.

First we need to recall a constructive definition of the extremal process associated with the d.f.~$F$ (cf.\ Ch.~4 in~\cite{Re87}). Let $\mathscr{P}$ be a Possion point process on $\R^2$ with the intensity measure
$\la$ specified by
\[
\la\bigl((a,b)\times (x,\infty)\bigr)=(a-b)\ln F(x), \quad a<b,\ x>x_0,
\]
and $\la\bigl(\R\times (-\infty,x_0]\bigr)=0,$ where $x_0:=\inf\{x\in\R: F(x) >0\}$ and we assume for simplicity that $F(x_0)=0$ if $x_0>-\infty$ (which is no loss of generality as we are interested in the behavior of the  upper records only).   Introduce the following notation:
\begin{equation}
\label{L}
L(B):= \inf\bigl\{x\in \R:  \mathscr{P}\bigl(B\times (x,\infty)\bigr)=0\bigr\}, \quad B\subset \R.
\end{equation}
The (continuous--time) extremal process $\{\mathscr M_t\}_{t> 0}$ is then defined by
$
{\mathscr M_t}:= L((0,t]), $ $t> 0.$ Clearly,
\begin{equation}
\label{embed}
{\mathscr M_{s_n}}=\bigvee_{j=1}^n \Xc_j^*, \quad n\ge 1,
\end{equation}
where  $ {\mathscr \Xc_j^*}:=L\bigl( (s_{j-1}, s_j] \bigr) $, $ j\ge 1,$ are independent r.v.'s,
\[
\pr (\Xc_j^*\le x) = \pr \bigl(\mathscr{P}  \bigl((s_{j-1},s_j]\times (x,\infty)\bigr)=0\bigr)=
e^{ \al_j\ln F(x)}=F^{\al_j} (x), \quad x\in\R,
\]
for $j\ge 1$ (and likewise $\pr (\mathscr{M}_t\le x)=F^t (x),$ $t>0$), so that $\{\Xc_n^*\}_{n\ge 1} \deq \{\Xc_n\}_{n\ge 1}$ and hence $\{\Mc_n\}_{n\ge 1} \deq \{\mathscr M_{s_n}\}_{n\ge 1}.$ This remarkable embedding of the partial maxima sequence for  an $F^\al$-scheme into a continuous--time extremal process is one of the powerful tools for solving problems related to such schemes, see e.g.~\cite{BaRe87}. It proved to be very handy in our case as well.

Let $b_t,$ $t\ge 0,$ be a non-decreasing right-continuous real-valued  function. As in the discrete-time case, we will use
$
\{\mathscr{M}_t > b_t {\rm\ ev.}\}
$
to denote the event that $\mathscr{M}_t>b_t$ eventually, i.e.\ that $\sup\{t>0: \mathscr{M}_t \le b_t \}<\infty$. We do not assume in the following theorem that $F$ is continuous.

\begin{theo}\label{th2}
\begin{enumerate}
	\item[{\rm (i)}]
		One always has $\pr (\mathscr{M}_t > b_t {\rm \ ev.})=0$ or~$1. $	 

\item[\rm (ii)]  If $g_t:=1-F(b_t)\to c >0$ as $t\to\infty$ then $\pr (\mathscr{M}_t > b_t  {\rm\ ev.})=1.$
	
\item[\rm (iii)]  If $\lim_{t\to 0}g_t = 0$ and $\liminf_{t\to\infty}  tg_t <\infty$   then $\pr (\mathscr{M}_t > b_t  {\rm\ ev.})=0.$
	
\item[\rm (iv)]  If $\lim_{t\to 0}g_t = 0$ and $\lim_{t\to\infty}  tg_t =\infty$   then
\begin{align}
\label{J}
\pr (\mathscr{M}_t > b_t {\rm\ ev.})=
\left\{
\begin{array}{ll}
0 & \mbox{if}\ \   J(b) =\infty,\\
1 &  \mbox{if}\ \  J(b)  <\infty,
\end{array}
\right.
\quad \mbox{where}\   J(b):= \int_0^\infty g_t e^{-t g_t}dt.
\end{align}
\end{enumerate}
\end{theo}

\begin{rema}\label{re1}
\rm Note that the assumption that $b_t$ is non-decreasing  does not actually restrict the generality. Thus, for part~(iv), arguing as on p.~382 in~\cite{Kl84}, it is not hard to verify that if $b_t$ is a general real-valued function such that $ g_t\to 0,$ $tg_t\to\infty$ as $t\to\infty,$ then, for some $t_0<\infty$, the non-decreasing function 
\[
\overline{b}_t:=
\left\{
\begin{array}{ll}
b_{t_0}, & t\in (0,t_0),\\
\sup_{t_0\le u\le t} b_u,&  t\ge t_0, 
\end{array}
\right.
\]
has the property that  $\pr (\mathscr{M}_t > b_t {\rm\ ev.})=\pr (\mathscr{M}_t > \overline{b}_t {\rm\ ev.}).$
\end{rema}

Theorem~\ref{th2} enables us to give a complete characterization of the situations where the  condition $\pr (\Mc_n > \ell_n {\rm \ ev.})=1$ from Theorem~1 is satisfied.

\begin{coro}\label{co2}
The following assertions hold true for an $F^\al$-scheme under condition~\Con\ and a non-decreasing sequence $\{\ell_n\}$.
\begin{enumerate}
		\item[{\rm (i)}]
		One always has $\pr (\Mc_n > \ell_n {\rm \ ev.})=0$ or~$1. $
		
	\item[{\rm (ii)}]  If $q_n:=1-F(\ell_n)\to c>0$ as $n\to\infty$  then $\pr (\Mc_n > \ell_n {\rm \ ev.})=1. $

\item[{\rm (iii)}] If $\lim_{n\to\infty}q_n= 0$ and $\liminf_{n\to \infty} s_n q_n< \infty$ then $\pr (\Mc_n > \ell_n {\rm \ ev.})=0. $

\item[{\rm (iv)}] If $\lim_{n\to\infty}q_n= 0$ and $\lim_{n\to \infty} s_n q_n=\infty$ then
\end{enumerate}
\[
\pr (\Mc_n > \ell_n {\rm \ ev.})=
\left\{
\begin{array}{ll}
0 & \mbox{if}\ \   K(\ell) =\infty,\\
1 &  \mbox{if}\ \  K(\ell)  <\infty,
\end{array}
\right.
\quad \mbox{where}\  K(\ell):=
\sum_{n\ge 1}  e^{-s_nq_n} (1- e^{-\al_{n+1}q_n}).
\]
\end{coro}

\begin{rema}
	{\rm Note that without condition~\Con\ the assertions of the theorem do not need to hold. Indeed, if $s_\infty:=\lim_{n\to\infty} s_n<\infty $ then there are only finitely many records in the sequence  $\boldsymbol{\Xc}$ and $\Mc_\infty:=\lim_{n\to\infty} \Mc_n\deq \mathscr{M}_{ s_\infty  }$ is a proper non-degenerate r.v.  Hence for, say, a constant sequence  $\ell_n\equiv \ell$ one would have $\pr (\Mc_n > \ell_n {\rm \ ev.})= \pr (\mathscr{M}_{ s_\infty }>\ell )=F^{s_\infty} (\ell), $ which can  be neither~0 nor~1 etc.}
\end{rema}

Finally, we will address the question concerning  the  dependence of the record indicators in threshold $F^\al$-schemes. As was pointed out earlier, in the general case, even when the $X_n$'s are independent,  the nature  of dependence between record indicators  is very complicated and difficult to describe. However, we were able to obtain the following asymptotic pair-wise independence result    in the case of independent $X_n$'s in a threshold $F^\al$-scheme with a flat threshold $\ell_n=\ell ,$ $n\ge 1.$  It is actually  possible to extend the result of Theorem~\ref{th3}  to assert that, say, $\pr(I_n=1|I_{m_1}=1, I_{m_2}=1)$ and $ \pr(I_n =1)$ are asymptotically equivalent as $n>m_2>m_1\ge k\to\infty$ etc, but the set of conditions  for such an assertion will already  be  quite cumbersome.

\begin{theo}\label{th3}
Assume that the r.v.'s in the sequence  $\boldsymbol{X}$ are independent,  condition  \Conn\ is satisfied for a constant threshold sequence  $\ell_n=\ell $ and $\de_n=0,$  $n\ge 1 $. If
\begin{equation}
\label{con_s_al}
\lim_{k\to \infty}\frac{	s^2_k}{\al_k }(\al_k  \vee 1) h^{s_k }= 0, \quad\mbox{where}\quad  h:=F(\ell),
\end{equation}
	then
\begin{equation}
\label{to_1}
\sup_{n>m\ge k}\biggl|
\frac{\pr(I_n=1|I_m=1)}{ \pr(I_n =1)}-1\biggr|\to 0\quad as \quad k\to \infty.
\end{equation}
\end{theo}

\smallskip

\begin{rema}{\rm
We can somewhat extend the conditions of the last theorem by assuming that we have a sequence $\boldsymbol{X}^{(k)}=\{X_n^{(k)}\}_{n\ge 1},$ $k\ge 1,$ of threshold $F^\al$-schemes indexed by the parameter  $k\to\infty$ and that these schemes share a common d.f.~$F$ and sequence~$\bal,$ but have  growing thresholds $\ell=\ell^{ (k)}$ such that $h=h^{ (k)}:=F(\ell^{ (k)})\to 1$. Then relation~\eqref{to_1} will still hold true under   assumption~\eqref{con_s_al} with $h=h^{(k)}.$ We just note here that the upper bound for the last term  in~\eqref{ssaa2} in the proof of Lemma~\ref{le2} will be valid when $(h^{(k)})^{s_k }<e^{-2}$ and that~\eqref{con_s_al} (with $h=h^{(k)}$) ensures that $(h^{(k)})^{s_k }\to 0 $ as $k\to\infty.$
}
\end{rema}

\section{Proofs}

\begin{proof}[Proof of Theorem~\ref{th1}] To construct the desired coupling, we will start with a sequence of i.i.d.\ uniform-$(0,1)$ r.v.'s $U_1, U_2, \ldots$ given on some  probability space, set $X_1:=F_1^{\leftarrow} (U_1),$
\begin{equation}
\label{Xn}
 X_n:=F_{n|X_1,\ldots,X_{n-1}}^\leftarrow (U_{n })\quad\mbox{for}\quad n\ge 2,
\end{equation}
and let $\Xc_n :=(F^{\al_n})^\leftarrow (U_n)=F^\leftarrow (U_n^{1/\al_n})  ,$ $n\ge 1.$
These sequences of r.v.'s will clearly have the desired distributions.

Next introduce the events
\begin{align*}
A_n
 & := \{X_n>\ell_n\}= \{U_n>F_{n|X_1,\ldots,X_{n-1}}(\ell_n)\}
 =\{U_n>F_n (\ell_n)\},
 \\
\Ac_n &
:= \{\Xc_n>\ell_n\}= \{U_n>F^{\al_n}(\ell_n)\}
\end{align*}
for $n\ge 2,$ where we   used condition~\Conn(ii)  for the last equality in the first line. Observe that, for the symmetric difference of these events, one has
\[
\pr (A_n\triangle  \Ac_n)
  = |F_n (\ell_n)- F^{\al_n}(\ell_n)|\le \de_n.
\]
Therefore, by the Borel--Cantelli lemma and condition     \Conn (iii), these events occur finitely often a.s., so that
\[
T_1:  = \inf\Bigl\{n\ge 1: \sum_{k\ge n}  \ind ( A_k \triangle \Ac_k)=0 \Bigr\}
 <\infty \quad\mbox{a.s.}
\]
Now introduce  the events $B_n:= A_n\cap \Ac_n=\{U_n>F_n (\ell_n) \vee F^{\al_n} (\ell_n)=:l_n\},$ $n\ge 1.$ Clearly,
\[
d_{TV} \bigl(F_n |_{F_n^{\leftarrow}(l_n)}, F^{\al_n}|_{(F^{\al_n})^{\leftarrow}(l_n)}\bigr) \le \de_n.
\]
Therefore, using Dobrushin's maximal coupling theorem~\cite{Do70},  we can recursively re-define for $n\ge 2$ the r.v.'s $X_n$ and $\Xc_n$ on the set $B_n$ only (extending the underlying probability space if necessary and updating at each step the definitions~\eqref{Xn} for the ``later" observations  $X_{n+1}, X_{n+2},\ldots$ accordingly) in such a way that
\[
\pr (X_n\neq \Xc_n; B_n)\le \de_n.
 \]
Again by the Borel--Cantelli lemma  and conditions~\Conn(iii),(iv), one has
\[
T_2:  = \inf\Bigl\{n \ge T_1: \sum_{k\ge n} \bigl( \ind    (X_k\neq \Xc_k; B_k)+ \ind (\Mc_k\le \ell_k)\bigr)=0 \Bigr\}
<\infty \quad\mbox{a.s.}
\]
Now for $n\ge T_2$ we always have $ \Mc_n >  \ell_n$ and whenever one of the events $A_n, \Ac_n$ occurs, the other one occurs as well and $X_n=\Xc_n$ for that~$n$.

In view of condition~\Conn(i), one has $F (M_{T_2})<1$ a.s. Therefore, by virtue of~\Con, with probability~1 there exists an~$n\ge T_2$ such that $\Xc_n=\Mc_n> M_{T_2}$. From the definition of~$T_2$, one then also has $\Xc_n >\ell_n$ and  $X_n= \Xc_n,$ so that  $M_n> M_{T_2}$. We conclude that
\[
T_3:= \inf\{n> T_2: M_n> M_{T_2},  \Mc_n> \Mc_{T_2}, \Ic_n=1 \}<\infty \quad\mbox{a.s.}
\]
	
We claim that  $(M_n, I_n)=(\Mc_n, \Ic_n)$ for all $n\ge T_3.$ Indeed, assume that  $\Ic_n=1$ for an $n\ge T_3$. Then $\Xc_n=\Mc_n> \ell_n $ and so also $X_n=\Xc_n> \ell_n .$ Now if~$I_n=0$ then there exists a $k\in [T_2, n] $ such that $X_k>X_n>\ell_n\ge \ell_k$, which implies that $X_k=\Xc_k $ (by the definition of~$T_2$) and so $\Xc_k>\Xc_n$, which contradicts the assumption that $\Ic_n=1$. A symmetric argument shows that if $I_n=1$ then also $\Ic_n=1$, thus proving  that $I_n= \Ic_n $ for all  $n\ge T_3.$ Now, whenever $\Xc_n$ with $n\ge T_2$ is a record we must have $\Xc_n=\Mc_n>\ell_n,$ which ensures that $X_n=\Xc_n$. Since  we already know that    at such times $n\ge T_3$ there will be a record in $\boldsymbol{X}$ as well, we obtain that $M_n=X_n=\Xc_n=\Mc_n$. Theorem~\ref{th1} is proved.
\end{proof}

\begin{proof}[Proof of Theorem~\ref{th2}]
(i)~Recalling notation~\eqref{L}, set
\[
Y_n:=L((n-1,n]), \quad n\in \Z, \qquad M^Y_m:=\bigvee_{k=1}^m Y_k, \quad m\ge 1.
\]
By construction, the $Y_n$'s are i.i.d.\ with d.f.~$F$. As both $\mathscr{M}$ and $b$ are non-decreasing, we have (using $\lfloor t\rfloor$ for the integer part of $t$)
\begin{align}
\{M^Y_n > b_{n+1} {\rm\ ev.}\}
  = \{\mathscr{M}_t > b_{\lfloor t\rfloor +1} {\rm\ ev.}\}
& \subset
\{\mathscr{M}_t > b_t {\rm\ ev.}\}
\notag \\
 &
\subset
\{\mathscr{M}_t > b_{\lfloor t\rfloor} {\rm\ ev.}\}
= \{M^Y_n > b_{n} {\rm\ ev.}\}.
\label{MM}
\end{align}

We can assume w.l.o.g.\  that there is no $x_0$ such that $F(x_0-)<F(x_0)=1,$ as if such a point existed then one would have $M^Y_n=x_0$ ev.\ (and hence $\mathscr{M}_t=x_0$  ev.) a.s., making assertion~(i) obvious. With that assumption, clearly $\pr (A)=1$ for $A:=\bigl\{\sum_{k=1}^\infty \ind(Y_{k+1} >M_k)=\infty\bigr\}$ and, for any $k\ge 1,$
\begin{align*}
A\{M^Y_n > b_{n} {\rm\ ev.}\}
 & = A \bigl\{ \mbox{$\bigvee_{j=k}^n$} Y_j > b_n  {\rm\ ev.}\bigr\}
  \\
  & = \biggl\{\sum_{m=k+1}^\infty \ind\bigl(Y_m > \mbox{$\bigvee_{j=k}^{m-1}$}Y_j \bigr)=\infty;
\    \mbox{$\bigvee_{j=k}^n$} Y_j > b_n  {\rm\ ev.} \biggr\}.
\end{align*}
Therefore the event on the left-hand side belongs to the tail $\si$-algebra for $\{Y_n\}_{n\ge 1}.$ We conclude that, by Kolmogorov's 0--1 law, its probability must be either~0 or~1  and that  the same applies to  $\pr (M^Y_n > b_{n} {\rm\ ev.})$ as well since  $\pr (A)=1$. The same argument is valid for $\pr (M^Y_n > b_{n+1} {\rm\ ev.})$.

It remains to observe that  $\{M_n\}_{n\ge 2} \deq \{Y_0\vee M_{n-1}\}_{n\ge 2}$ and so
\begin{align}
\pr (M^Y_n > b_{n} {\rm\ ev.})
& = \pr (Y_0\vee M^Y_{n-1} > b_{n} {\rm\ ev.})
 = \pr \bigl(A\{ Y_0\vee M^Y_{n-1} > b_{n} {\rm\ ev.}\}\bigr)
  \notag
\\
&
 =
  \pr \bigl( A\{ M^Y_{n-1} > b_{n} {\rm\ ev.}\}\bigr)
   = \pr (  M^Y_{n} > b_{n+1} {\rm\ ev.}  ).
   \label{n_n+1}
\end{align}
Assertion~(i) is proved, as the probabilities of the left-hand and right-hand sides in~\eqref{MM} are equal to each other, their only possible values being~0 and~1.

\medskip
(ii)~This assertion  is obvious as $\mathscr{M}_t$ is non-decreasing, $\mathscr{M}_n \ge Y_n,$ $n\ge 1,$ and so, setting $B:=\bigcup_{n\ge 1} (-\infty, b_n] $ (which can be an open or closed half-line, $\pr (Y_j\in B)=1-c<1$), one has
$
\pr (\mathscr{M}_t > b_t {\rm\ ev.}) \ge
 \pr \bigl(  \bigcup_{n\ge 1} \{Y_n\in B^c\}\bigr)
  = 1- \pr \bigl(  \bigcap_{n\ge 1} \{Y_n\in B\}\bigr)=1.
$

\medskip
(iii)~In this case, there exists a sequence $t_n\to\infty $ as $n\to \infty$ such that $t_n g_{t_n}<c<\infty$, $n\ge 1.$ One has
\begin{align*}
\pr (\mathscr{M}_{t_n} \le b_{t_n} {\rm\ i.o.})
 & = \lim_{k\to\infty}
 \pr \bigl( \mbox{$\bigcup_{n\ge k}$ }  \{ \mathscr{M}_{t_n} \le b_{t_n}\} \bigr)
 \\
 & \ge \liminf_{n\ge 1} \pr \bigl( \mathscr{M}_{t_n} \le b_{t_n} \bigr)
  =
  \liminf_{n\ge 1} F^{t_n}(b_{t_n})
  \\
  &
  =   \liminf_{n\ge 1} e^{- t_n  g_{t_n}(1+o(1))}
  \ge  e^{-c} >0.
\end{align*}
Hence $\pr (\mathscr{M}_t > b_t  {\rm\ ev.})<1.$  In view of~(i), that probability must be~0.

\medskip
(iv)~We will make use of the following criterion from~\cite{Kl84, Kl85} derived for an i.i.d.\ sequence $\{Y_n\}$ in the case  a non-decreasing sequence $\{b_n\}$ when $g_n= 1- F(b_n)\to 0,$ $n g_n\to\infty$ as $n\to\infty$:
\begin{equation}
\label{Klass}
\pr (M^Y_n > b_n {\rm\ ev.})=
\left\{
\begin{array}{ll}
0 & \mbox{if}\ \    \Sigma (b) =\infty,\\
1 &  \mbox{if}\ \  \Sigma(b)  <\infty,
\end{array}
\right.
\quad \mbox{where}\quad  \Sigma (b):= \sum_{n\ge 1} g_n e^{-n g_n}.
\end{equation}
As the function $f(x)= x e^{-ax}$ with $a>0$ is decreasing for $x\ge a^{-1}$, $g_n=o(1)$ as $n\to\infty,$ and the function  $g_t$ is non-increasing, we  have,  for $t\in [n,n+1]$ and all sufficiently large~$n$,
\begin{align*}
g_n e^{- n g_n}(1+o(1))& = g_n e^{-(n+1) g_n} \le g_t e^{-(n+1) g_t}
 \le g_t e^{-t g_t}
 \\
 & \le g_t e^{-n g_t} \le g_{n+1} e^{-n g_{n+1}}
  =
 g_{n+1} e^{-(n+1) g_{n+1}} (1+o(1)) .
\end{align*}
Therefore the integral $J(b)$ in~\eqref{J} and the sum  $\Sigma (b)$ in~\eqref{Klass} converge (diverge) simultaneously. It remains to make use of relations~\eqref{MM} and~\eqref{n_n+1}. Theorem~\ref{th2} is proved.
\end{proof}

\begin{proof}[Proof of Corollary~\ref{co2}.] Setting $s_0:=0,$ introduce the
non-decreasing function
\begin{align}
\label{bt}
b_t:= \sum_{n\ge 1} \ell_n \ind (t\in [s_{n}, s_{n+1})), \quad t>0.
\end{align}
In view of the embedding~\eqref{embed} of $\{\Mc_n\}$ in the extremal process $\{\mathscr{M}_t\}$ associated with the d.f.~$F$ and the special form of the boundary $b_t,$ we clearly have  $\pr (\Mc_n >\ell_n {\rm \ ev.})=\pr (\mathscr{M}_t > b_t {\rm \ ev.})$. Therefore part~(i) follows from Theorem~\ref{th2}(i). It remains to verify that the conditions from parts~(ii)--(iv) of the corollary are equivalent to the conditions from the respective parts (ii)--(iv) of Theorem~\ref{th2} in the case of our boundary~\eqref{bt}.

Part~(ii) is obvious. Part~(iii) is also obvious since, for the boundary~\eqref{bt}, one has  $t g_t = s_n q_n$ for $t=s_n$ and $s_n\to \infty$ as $n\to\infty $  by~\Con.  Finally, in part~(iv)
one has
\begin{align*}
J(b) = \sum_{n\ge 1}  \int_{s_{n }}^{s_{n+1}}
 q_n e^{-t q_n} dt
 =\sum_{n\ge 1}  e^{-s_nq_n} (1- e^{-\al_{n+1}q_n}).
\end{align*}
Corollary~\ref{co2} is proved.
\end{proof}

 \begin{proof}[Proof of Theorem~\ref{th3}] We will split the proof of the theorem into three lemmata.

\begin{lemo}
	\label{le0}
	Under conditions of Theorem~\ref{th3}, one can construct $\boldsymbol{X}$ and $\boldsymbol{\Xc}$ so that, for any $m\ge 1,$  on   the event $\{M_m>\ell\}$ one has
\begin{align}
( M_n, I_n) =(\Mc_n, \Ic_n), \quad \mbox{for}\  n\ge m.
\end{align}
\end{lemo}
	
\begin{proof} Under conditions of Theorem~\ref{th3}, we can construct  $\boldsymbol{X}$ and $\boldsymbol{\Xc}$ by letting $
	X_n:=F_n^{\leftarrow} (U_n),$ 	$\Xc_n:=(F^{\al_n})^{\leftarrow} (U_n),$ $n\ge 1
$, where
$\{U_n\}_{n\ge 1}$ is a sequence of independent uniform-$(0.1)$  r.v.'s.  Then clearly
\begin{align}
\label{XXX}
\{X_n>\ell\}= \{\Xc_n>\ell\},\quad
X_n \ind (X_n>\ell) = \Xc_n \ind (\Xc_n>\ell), \quad n\ge 1.
 \end{align}
Hence, for $n\ge m\ge 1,$
\begin{align*}
\{I_n=1, M_m>\ell\}
 & =\{I_n=1, M_m>\ell, X_n >\ell\}
  \\
  & = \{ X_n >\ell\}\cap \biggl[\bigcup_{k\le m}\{ X_k>\ell\}\biggr] \bigcap_{r<n}
   \bigl[ \{X_r \le \ell\}\cup \{\ell< X_r <X_n\}\bigr]
 \\
 & =  \{ \Xc_n >\ell\}\cap \biggl[\bigcup_{k\le m}\{ \Xc_k>\ell\}\biggr] \bigcap_{r<n}
 \bigl[ \{\Xc_r \le \ell\}\cup \{\ell< \Xc_r < \Xc_n\}\bigr]
 \\
 &=\{\Ic_n=1, \Mc_m>\ell, \Xc_n >\ell\}
 =\{\Ic_n=1, \Mc_m>\ell \}.
 \end{align*}
The claim concerning the equality of  $M_n$ and $\Mc_n$ is next to obvious from~\eqref{XXX}: setting $\tau (m):=\min\{k\ge 1: X_k=M_m\}$ (which is the same as $\min\{k\ge 1: \Xc_k=\Mc_m\}$ on the event $\{M_m>\ell\}=\bigcup_{k\le m} \{X_k>\ell\}=\bigcup_{k\le m} \{\Xc_k>\ell\}=\{\Mc_m>\ell\}$), one has
\begin{align*}
M_n \ind (M_m>\ell)
 & = \bigvee_{k=1}^n X_k \ind (X_{\tau (m)} >\ell)
  = \bigvee_{k=1}^n X_k \ind (X_{k} >\ell) \ind (X_{\tau (m)} >\ell)
  \\
  & = \bigvee_{k=1}^n \Xc_k \ind (\Xc_{k} >\ell) \ind (\Xc_{\tau (m)} >\ell)
  = \Mc_n \ind (\Mc_m>\ell).
 \end{align*}
Lemma~\ref{le0} is proved.\end{proof}

Next observe that one clearly  has
 	\begin{align}
 	b_m & := \pr (I_m=1, M_m\le \ell)\le \pr (M_m\le \ell) = h^{s_m },\quad m\ge 1,
 	\label{b_for_b}\\
 	c_{m,n} & :=\pr (I_m=I_n=1, M_n\le \ell)
 	\le \pr (M_n\le \ell) = h^{s_n },\quad n>m\ge 1.
 	\label{b_for_c}
 	\end{align}

\begin{lemo}
	\label{le1}
	For $ n>m\ge 1,$
\begin{align}
 \pr(I_m =1)&= \frac{\al_m }{s_m }(1-h^{s_m })+b_m
 = \frac{\al_m }{s_m }+\theta_m,\quad |\theta_m|\le  h^{s_m } ,
\label{b_for_I}
\\
 \pr(I_m =I_n=1) & =
   \frac{\al_m \al_n }{s_n -s_m }
 \biggl(\frac{1-h^{ s_m }}{s_m }
 	- \frac{1-h^{ s_n }}{s_n }\biggr)
 	\notag\\
 	&\hskip 2 cm   +  b_m \al_n  \frac{1-h^{ s_n -s_m }}{s_n -s_m }+	 c_{m,n}.
\label{b_for_II}
\end{align}	
\end{lemo}

\begin{proof}
By Lemma~\ref{le0}, \eqref{a/s} and~\eqref{I_M},
\begin{align}
\pr(I_m =1)&= \pr(I_m =1, M_m >\ell) +b_m
   =\pr(\Ic_m =1, \Mc_m >\ell) +b_m\notag\\
  & =\pr(\Ic_m =1)\pr (\Mc_m >\ell) +b_m
  =\frac{\al_m }{s_m } (1-h^{s_m })   +b_m,
\label{***_2}
\end{align}	
which is the first equality in~\eqref{b_for_I}. The second one is obvious.

To establish~\eqref{b_for_II}, we set $A_{u,v}(x):= \bigcap_{u<r<v}\{\Xc_r<x \}  $ for $v>u\ge 0$  and  observe that, in view of Lemma~\ref{le0} and independence of the $\Xc_r$'s,
\begin{align*}
P_1&:=  \pr(I_m =I_n=1, M_m>\ell) = \pr(\Ic_m =\Ic_n=1, \Mc_m>\ell)\\
 & =  \int_{\ell <x_1 <x_2}
 \pr \bigl(A_{0,m}(x_1)A_{m,n}(x_2)
 \big|\Xc_m=x_1, \Xc_n=x_2\bigr)dF^{\al_m }(x_1)dF^{\al_n}(x_2)\\
 & =  \int_{\ell <x_1 <x_2}
 \biggl(\prod_{u=1}^{m-1}F^{\al_u }(x_1)\biggr)
 \biggl(\prod_{v=m+1}^{n-1}F^{\al_v }(x_2)\biggr)
 dF^{\al_m }(x_1)dF^{\al_n}(x_2)
 \\
  & = \int_{h <y_1 <y_2\le 1}
  y_1^{s(m-1)} y_2^{s_n  - s_m }
   dy_1^{\al_m } dy_2^{\al_n}
   \\
   & = \frac{\al_m \al_n }{s_n -s_m }
   \biggl(\frac{1-h^{s_m }}{s_m }
    -\frac{1-h^{s_n }}{s_n }\biggr).
\end{align*}
Next, by the independence of the $X_r$'s, one has
\begin{align}
P_2&:=  \pr(I_m =I_n=1, M_m\le \ell <M_n)\notag\\
&  = \pr\bigl(\{I_m = 1, M_m\le \ell\}\cap \{ X_r< X_n, m<r<n; X_n>\ell\}\bigr)\notag
\\
&  = \pr\bigl( I_m = 1, M_m\le \ell\bigr)\,
 \pr \bigl( X_r< X_n, m<r<n; X_n>\ell\bigr).
 \label{P2}
 \end{align}
The first factor of the right-hand side is $b_m,$ whereas the second one is of the form $\pr (I'_{n-m}=1, M'_{n-m}>\ell)$ for the shifted sequence $\{X'_j:= X_{j+m}\}_{j \ge 1},$ which satisfies the same assumptions as the original~$\boldsymbol{X}$. Therefore, similarly to the computation of the term $\pr(I_m =1, M_m >\ell)$ in~\eqref{***_2} and using the self-explanatory notations $\Ic'_j  ,$ $\Mc'_j$, we obtain  that the  second factor on the right-hand side of~\eqref{P2} is equal to
\[
\pr(\Ic'_{n-m} =1)\pr (\Mc'_{n-m} >\ell)
 = \frac{\al_n}{s_n -s_m } (1-h^{s_n -s_m }).
\]
We conclude that
$
P_2 =  b_m\al_n(1-h^{s_n -s_m })/(s_n -s_m ).
$
Now~\eqref{b_for_II} immediately follows from the obvious representation $ \pr(I_m =I_n=1) =P_1+P_2+c_{m,n}.$
\end{proof}

\begin{lemo}
	\label{le2}
As $k\to \infty,$
\begin{align}
\frac{s_m s_n }{\al_m \al_n }\pr(I_m =1)\pr(I_n =1)
 \to 1,
\quad \frac{s_m s_n }{\al_m \al_n }\pr(I_m =1,I_n =1) \to 1
\label{XX}
\end{align}	
uniformly in $n>m\ge k.$
\end{lemo}

 	\begin{proof}
The first relation in~\eqref{XX} is obvious from~\eqref{b_for_I} as the left-hand side of the former  relation equals
 \[
 \frac{s_m s_n }{\al_m \al_n }
 \biggl(\frac{\al_m }{s_m  }+\theta_m\biggr)
 \biggl(\frac{\al_n }{s_n  }+\theta_n\biggr)
= \biggl(1+ \theta_m\frac{s_m  }{\al_m } \biggr)
\biggl(1+ \theta_n\frac{s_n  }{\al_n } \biggr)
 \]
and $|\theta_r|\le h^{s_r }=o(\al_r /s_r )$ as $r\to \infty$ by virtue of~\eqref{con_s_al} and~\Con.

To establish the second relation in~\eqref{XX}, we will use~\eqref{b_for_II}. Recall that the first two terms on the latter relation's right-hand side we denoted by $P_i,$ $i=1,2.$ First note that, by the mean value theorem for the function $f(x):=xh^x,$ $h=\rm const$, one has
\begin{align}
\frac{s_m s_n }{\al_m \al_n }P_1
& =
\frac{s_m s_n }{s_n -s_m }
\biggl(\frac{1-h^{s_m }}{s_m }
-\frac{1-h^{s_n }}{s_n }\biggr)
\notag\\
&
= 1- h^{s_m }-h^{s_n } +\frac{s_n h^{s_n }-s_m  h^{s_m } }{s_n -s_m }
\notag\\
&
= 1- h^{s_m }-h^{s_n } +(1 +\ln h^{s})h^s
\label{ssaa1}
\end{align}
for some $s\in [s_m , s_n ].$ Clearly,  the expression in the last line in~\eqref{ssaa1} tends to~1 uniformly in $n>m\ge k\to\infty.$

Next, in view of~\eqref{b_for_b} and again by the mean value theorem (this time for the function $f(x) =h^{1/x},$ $h=\rm const \in (0,1)$), one has
\begin{align}
0 &\le   \frac{s_m s_n }{\al_m \al_n }P_2
=
\frac{b_m s_n s_m \bigl(1-h^{s_n -s_m }\bigr)}{\al_m (s_n -s_m )}
\notag \\
& \le
\frac{h^{ s_m }-h^{s_n }}{\al_m (1/s_m -1/s_n )}
=\frac{|\ln h |}{\al_m } s^2 h^s
\label{ssaa2}
\end{align}
for some $s\in [s_m , s_n ].$  As the function $s^2 h^s$ is decreasing for $s$ such that $h^s<e^{-2},$ the right-hand side of~\eqref{ssaa2} is bounded by $|\ln h |s^2_m h^{s_m }/ \al_m    $ for all sufficiently large~$m$. That expression clearly vanishes uniformly in $m\ge k\to\infty$ in view of condition~\eqref{con_s_al}.

Finally,
\begin{align*}
 \frac{s_m s_n }{\al_m \al_n } c_{m,n}
& \le
\frac{s_m s_n }{\al_m \al_n } h^{s_n }
 \le \frac{s_m  }{\al_m  } h^{s_m /2}
 \frac{ s_n }{ \al_n } h^{s_n /2} \to 0
\end{align*}
uniformly in $n> m\ge k\to\infty$, again due to~\eqref{con_s_al}. Lemma~\ref{le2} is proved.	\end{proof}
 	
It remains to observe that~\eqref{to_1} immediately follows from relations~\eqref{XX}. Theorem~\ref{th3} is proved.  \end{proof}
	
	\section*{Acknowledgments}
A major part of the research presented in this paper was done when K.~Borovkov
was visiting    the School of Informatics and Data Science, Hiroshima University, whose hospitality and support  is gratefully acknowledged. P.~He's work on this research was   supported  by the Maurice H.~Belz scholarship.

\end{document}